\documentclass[12pt,letterpaper]{amsart}
\usepackage{amsmath, amsthm, amssymb}
\usepackage[tmargin=1.4in,bmargin=1.2in,rmargin=1.4in,lmargin=1.4in]{geometry}
\usepackage[breaklinks=true,hidelinks]{hyperref}

\usepackage{comment,color}

\theoremstyle{plain}
\newtheorem{theorem}{Theorem}[section]
\newtheorem{lemma}{Lemma}[section]
\newtheorem{prop}{Proposition}[section]
\newtheorem{cor}{Corollary}[section]

\theoremstyle{definition}
\newtheorem{defin}{Definition}[section]

\DeclareMathOperator{\Aut}{\ensuremath{Aut}}
\newcommand{\Fr}{\mathrm{Fr}}

\begin{document}

\title[On the Category $\mathcal{O}$ for Generalized Weyl Algebras]{On the Category $\mathcal{O}$ for Generalized Weyl Algebras}
\author{Ruben Mamani-Velasco, Akaki Tikaradze}
\email{rubenlimbert.mamanivelasco@rockets.utoledo.edu}
\email{Akaki.Tikaradze@utoledo.edu}
\address{University of Toledo, Department of Mathematics \& Statistics, 
Toledo, OH 43606, USA}

\begin{abstract}
Let $H(R, \phi, z)$ denote a generalized Weyl algebra  associated with a ring
$R$, a central element $z\in Z(R)$  and an automorphism $\phi.$ Assume that for some $l \geq 1$
there exists a nilpotent element $b\in Z(R)$  fixed by $\phi,$ such that
$(z,\phi^i(z))=Z(R)$  for all $0<i<l$  and $\phi^l(z)=z$ modulo $b.$ Then we prove that the
category $\mathcal{O}$  over $H(R,\phi,z)$ is equivalent to the category
$\mathcal{O}$ over its $l$-th twist---the generalized Weyl
algebra $H(R,\phi^l,z).$ This result is significantly more general than the corresponding one
for the Weyl algebra over $\mathbb{Z}/p^n\mathbb{Z}.$ 
\end{abstract}

\maketitle

\section{Introduction}

Recall that Kashiwara's classical theorem on algebraic $D$-modules
states that given a smooth subvariety $Y$ of a smooth algebraic variety
$X$ over an algebraically closed field $\bold{k}$ of characteristic
$0$, the category of (left) $D$-modules over $X$ supported on
$Y$ is equivalent to the category of $D$-modules over $Y$. As a basic
application of this result, taking $X=\mathbb{A}^1_{\bold{k}}$  to be the
affine line and $Y=\lbrace 0\rbrace$ the origin, we obtain a well-known
statement that the category of left modules over the first Weyl algebra
\[
A_1(\bold{k})=\bold{k}\langle x, y\rangle/(xy-yx-1)
\]
on which $x$ acts locally nilpotently (the category $\mathcal{O}$  over
$A_1(\bold{k})$) is equivalent to the category of $\bold{k}$-vector
spaces, the functor (in one direction) being the functor of flat sections
$$M\mapsto M^x=\lbrace m\in M: xm=0\rbrace.$$ 

The picture is quite a bit more complicated and the above results no
longer hold if $\bold{k}$  is a field of positive characteristic (or more
generally, $\bold{k}$ contains $\mathbb{Z}/p^n\mathbb{Z}$).
Here we briefly recall a result of Shiho \cite{S}, who proved that given
a smooth variety $X$ over a ring $\bold{k}$  containing
$\mathbb{Z}/p^n\mathbb{Z}$ (under a suitable assumption on existence of a
lift of the Frobenius), there exists an equivalence between
the category of quasi-coherent sheaves on $X$ with
integrable $p$-connection and the category of quasi-coherent sheaves on
$X$ with integrable connection; this generalizes a
fundamental result by Ogus and Vologodsky \cite{OV}. We
now spell out this result for the case of the (first) Weyl algebra,
i.e.\ when $X$ is the affine line.

First, we use the following notation/terminology. Let $\bold{k}$ be a
commutative unital ring and $h\in \bold{k}$. The $h$-Weyl algebra over
$\bold{k}$ is denoted by $A_{(1,h)}(\bold{k})$ and is defined as
\[
\bold{k}\langle x, y\rangle/(xy-yx-h).
\]
Then, in the case of the affine line, the result of Shiho mentioned above
can be stated as follows.

\begin{theorem}\label{Shiho}
Let $\bold{k}$ be a commutative unital ring in which a prime
$p$ is nilpotent.
The category of $A_1(\bold{k})$-modules with locally nilpotent $x$-action
is equivalent to the category of $A_{(1,p)}(\bold{k})$-modules with
locally nilpotent $x$-action.
Moreover, any such $A_1(\bold{k})$-module $M$ restricted to $\bold{k}[y]$
is of the form $\bold{k}[y]_{\Fr}\otimes_{\bold{k}[y]}N$, where $N$ is an
$A_{(1,p)}(\bold{k})$-module of the above type, and
$\Fr:\bold{k}[y]\to\bold{k}[y]$ is the $\bold{k}$-Frobenius homomorphism
$\Fr(y)=y^p.$
\end{theorem}

Our main result is a vast generalization of this theorem to a wide class of algebras called generalized Weyl algebras. Recall their definition next.

\begin{defin}
Let $R$ be a ring, let $\phi\in \Aut(R)$ and let $z\in R$ be a central
element. Then the corresponding generalized
Weyl algebra $H(R, \phi, z)$ is the algebra generated over $R$ by
generators $x, y$ subject to the following relations:
\[
yx=z, \quad xy=\phi^{-1}(z), \quad xr=\phi^{-1}(r)x, \quad
yr=\phi(r)y.
\]
The category of left $H(R, \phi, z)$-modules on which $x$ acts locally
nilpotently is denoted by
$\mathcal{O}(H(R, \phi, z))$ and is called the category $\mathcal{O}.$ 
\end{defin}

We remark that some references use the inverse automorphism in their definition; our notation is used consistently throughout. 

Generalized Weyl algebras were introduced by Bavula \cite{B} and
include many classes of noncommutative algebras, such as quantized Weyl
algebras, the enveloping algebra $U(\mathfrak{sl}_2)$, and noncommutative
deformations of type A Kleinian singularities \cite{H}. We note that the definition of $\mathcal{O}(H(R, \phi, z))$ is
in analogy with that of the category $\mathcal{O}$ for
semisimple Lie algebras. An analogue for the category $\mathcal{O}$ for
a large class of generalized Weyl algebras (which is a full subcategory
of our $\mathcal{O}(H(R, \phi, z))$) was studied in
\cite{KT}.

Lunts and Rosenberg \cite{LR} proved that if $(z, \phi^i(z))=R$ for
all $i>0$, then $\mathcal{O}(H(R, \phi, z))$ is equivalent to the
category of $R/(z)$-modules.

The following is the main result of the paper. In what follows, given an
$R$-module $M$ and a central element $z\in Z(R)$, we denote
by $M^{z^\infty}$ the submodule of $M$ consisting of
elements annihilated by some power of $z.$ 

\begin{theorem}\label{main}
Let $R$ be a unital ring and $b\in Z(R)$ be a nilpotent element.
Let $z\in Z(R)$ and $\phi\in \Aut(R)$ such that $\phi(b)=b$ and for some $l\geq 1$ the following holds
\[
\phi^l(z)-z\in bZ(R), \quad (z, \phi^i(z))=Z(R), 1\leq i\leq l-1.
\]
Then the following functor is an equivalence of categories
\[
F:\mathcal{O}(H(R, \phi,z))\to \mathcal{O}(H(R, \phi^l, z)),\quad
F(M)=M^{z^\infty}.
\]
Moreover, if $M\in \mathcal{O}(H(R, \phi,z)),$ then there exists $N\in
\mathcal{O}(H(R, \phi^l, z))$ such that $M$ viewed as a module over
$R[y]$ is isomorphic to $R[y]_{\Fr}\otimes_{R[y]}N$. Here, for a GWA
$H(R,\psi,z)$, we denote by $R[y]$ its subring generated by $R$ and $y$
(a twisted polynomial ring), and $\Fr$ is the $l$-th ``Frobenius''
homomorphism from the subring $R[y]$ of $H(R,\phi^l,z)$ to the subring
$R[y]$ of $H(R,\phi,z)$, defined by $\Fr(y)=y^l$, $\Fr(r)=r$, $r\in R$.
\end{theorem}

Taking $R=\bold{k}[h]$, $\phi(h)=h-1$ and $z=h$, the algebra $H(R, \phi, z)$ can be
identified with the Weyl algebra $A_1(\bold{k}),$ while $H(R, \phi^p, z)$
is $A_{(1,p)}(\bold{k}).$ Here $l=p$ and $b=p\in\bold{k}$. Thus, we recover Theorem \ref{Shiho} from our main result.

\section{The proof}

First, we recall the following basic properties of generalized Weyl
algebras---henceforth abbreviated as GWAs.

\begin{prop}\label{GWA-basic properties}
Given a GWA $H(R, \phi, z)$,
\[
H(R,\phi,z)=R\oplus \bigoplus_{n\geq1}Rx^n\oplus
\bigoplus_{n\geq1}Ry^n=R\oplus \bigoplus_{n\geq 1}x^nR\oplus
\bigoplus_{n\geq 1}y^nR,
\]
and each summand is a free $R$-module of rank one.
\end{prop}

The decomposition is the usual $\mathbb{Z}$-grading (see \cite[Section 1.1]{B}),
and the freeness of the summands follows from \cite[Ch.~II, Lemma~3.1.6]{R};
although stated there for commutative $R$, the proof only uses that $z$ is central.

For the remainder of this section we adopt the assumptions and
notation from Theorem \ref{main}. We start with the
following. 

Define $\tau = y^l x^l$; then $\tau$ is central in $R$, and
moreover, $\tau = z \phi(z) \cdots \phi^{l-1}(z)$.
Denote by $e'$ the idempotent in $Z(R)/(\tau)$ corresponding to the
projection onto $Z(R)/(z)$ in the isomorphism given by the Chinese
remainder theorem 
\[
Z(R)/(\tau)\cong (Z(R)/(z)) \times (Z(R)/(\phi(z))) \times
\cdots \times (Z(R)/(\phi^{l-1}(z))).
\]

Now we are going to lift this idempotent to the following
completion of $Z(R).$ Let $\hat{Z}_{\tau}$ denote the completion of
$Z(R)$ with respect to $\tau$. Using Hensel's lemma, we
conclude that $e'$  admits a unique lift to $\hat{Z}_{\tau}$ which we
denote by $e.$ 

Denote by $\hat{R}_\tau$  the completion of $R$ with respect to  $(\tau)$  or
equivalently $(b, \tau),$ since $b$ is nilpotent. Replacing
$e$ by its image under the natural homomorphism $\hat{Z}_\tau\to \hat{R}_\tau$, we
may assume that $e$ is an element of the center of $\hat{R}_\tau$. The
idempotent $e$ plays a crucial role in the proof of our
main result. Replace also $e'$ by its image under the homomorphism
$Z(R)/(\tau)\to R/(\tau)$.  Thus, $e'$ is the unique central idempotent of $R/(\tau)$ with the property
that $(1-e')R/(\tau)=zR/(\tau).$
By abuse of notation, we also denote by $e'$ its image under the natural
quotient map $R/(\tau)\to R/(\tau,b)$.

Since $\phi^l(z)-z\in bZ(R)$, it follows immediately that $\phi(\tau)\equiv\tau\pmod b$. Hence $\phi$
descends to an automorphism of $R/(\tau, b)$, to be denoted by
$\bar{\phi}.$ It is clear that $\bar{\phi}^i(e')$ is the idempotent in
$R/(\tau, b)$  corresponding to the projection onto $R/(b, \phi^i(z))$; in particular, $\bar{\phi}^l(e')=e'$.
So,
$$\sum_{i=0}^{l-1}\bar{\phi}^i(e')=1,\quad  \bar{\phi}^i(e')\bar{\phi}^j(e')=0,\quad i\neq j\in \lbrace 0,\dots,
l-1\rbrace.$$

Next, since $\phi$ preserves the ideal $(b, \tau)$, we get
its completion automorphism, to be denoted by $\hat{\phi}$
henceforth. Again using the uniqueness of lifts of
idempotents modulo pro-nilpotent ideals, we easily see that
$\hat{\phi}^i(e)$ are lifts of $\bar{\phi}^i(e')$ (in particular, $\hat{\phi}^l(e)=e$) and give rise to an
orthogonal decomposition of $1$ in $\hat{R}_{\tau}$:
\[
\sum_{i=0}^{l-1}\hat{\phi}^i(e)=1,\quad
\hat{\phi}^i(e)\hat{\phi}^j(e)=0,\quad i\neq j\in \lbrace 0,\dots,
l-1\rbrace.
\]

The following lemma shows that any module in $\mathcal{O}(H( R, \phi,z))$
may (and will) be viewed as an object of
$\mathcal{O}(H(\hat{R}_\tau,\hat{\phi},z)).$

\begin{lemma}\label{category-O}
Let $M\in \mathcal{O}(H(R, \phi,z)).$ Then $\tau=y^lx^l\in Z(R)$
acts locally nilpotently on $M.$
\end{lemma}

\begin{proof}
First, recall the following equality that holds in any generalized Weyl algebra
\begin{equation}\label{equality}
y^nx^n=z\phi(z)\cdots\phi^{n-1}(z),\quad \text{for all}\quad n\in\mathbb{N}.
\end{equation}
Let $m\in M.$ We need to show that $\tau^km=0$ for some $k.$ Since
$\phi^l(z)\equiv z \pmod{b}$, we have that for all $n\in\mathbb{N}$
\[
y^{ln}x^{ln}=\prod_{i=1}^n (\tau+ba_i), \quad a_i\in Z(R).
\]
This implies that
\[
\tau^n=y^{ln}x^{ln}+\sum_{i=0}^{n-1}c_i\tau^i b^{n-i},
\qquad c_i\in Z(R).
\]
Let $b^t=0.$ Applying this identity recursively to the lower powers of $\tau$, we see that for any $n\geq t$ we may write $\tau^n$  as
$\sum_{i=n-t+1}^n b_iy^{li}x^{li}$  for some $b_i\in Z(R).$ Since $y^{k}x^{k}m=0$
for all $k\gg 0,$ we get that $\tau^km=0$ for $k\gg 0$ as desired.
\end{proof}

Hence given any module $M\in\mathcal{O}(H(R,\phi,z))$, it can be viewed as a left $\hat{R}_{\tau}$-module.
Conversely, any module in $\mathcal{O}(H(\hat{R}_\tau,\hat{\phi},z))$ may also
be viewed as an object of $\mathcal{O}(H( R, \phi,z))$ via the natural
homomorphism $R\to\hat{R}_\tau$. Thus, we obtain an equivalence of categories by extending the module
structures to the completion:
\[
\mathcal{O}(H( R, \phi,z))\simeq\mathcal{O}(H(\hat{R}_\tau,\hat{\phi},z)).
\]

The following simple observation will be very useful throughout the proof.

\begin{lemma}\label{units-idempotents}
Let $0\leq n<l$. Then $\hat{\phi}^n(e)\hat{\phi}^i(z)\in
(\hat{\phi}^n(e)\hat{R}_\tau)^*$  for all $0\leq i\neq n <l.$ Consequently, $ez=(e\tau)u$ where
$u$ is a unit in $e\hat{R}_\tau.$
Here $S^*$ denotes the group of units of a ring $S$.
\end{lemma}

\begin{proof}
Let $1=\alpha \phi^n(z)+\beta\phi^i(z)$  for some $\alpha, \beta\in R.$
Since the
image of $\hat{\phi}^n(e)$ in $\hat{R}_\tau/(\tau, b)=R/(\tau, b)$ is the idempotent corresponding
to the projection on $R/(\phi^n(z), b),$
we conclude that $\hat{\phi}^n(e)\hat{\phi}^n(z)\in (\tau, b).$ Hence
$\beta\hat{\phi}^n(e)\hat{\phi}^i(z)=\hat{\phi}^n(e)-\alpha\hat{\phi}^n(e)\hat{\phi}^n(z)$
and $\hat{\phi}^n(e)\hat{\phi}^n(z)$ is pro-nilpotent in $\hat{R}_\tau$  and
hence in $\hat{\phi}^n(e)\hat{R}_\tau$. This implies the desired result. Taking $n=0$ in the preceding argument,
we see that $e\hat{\phi}^i(z)$ is a unit in $e\hat{R}_\tau$ for
$1\leq i<l$. Since
\[
e\tau=ez\prod_{i=1}^{l-1}e\hat{\phi}^i(z),
\]
it follows that $ez=(e\tau)u$ for some $u\in(e\hat{R}_\tau)^*$. 
\end{proof}

The next result is crucial.

\begin{lemma}\label{Morita}
Put $\hat{H}=H(\hat{R}_\tau, \hat{\phi},z)$. Then $\hat{H}e\hat{H}=\hat{H}$ and
\[
e\hat{H}e=e\hat{R}_\tau\oplus \bigoplus_{n>0}
e\hat{R}_\tau x^{ln}\oplus\bigoplus_{n>0} e\hat{R}_\tau y^{ln}.
\]
Also, $\hat{H}e$ is a free right $e\hat{H}e$-module with a basis 
$\lbrace e, ye,\dots, y^{l-1}e\rbrace.$ 
\end{lemma}

\begin{proof}
We have that
$$y^nex^n=\hat{\phi}^n(e)y^nx^n=\hat{\phi}^n(e)z\hat{\phi}(z)\cdots
\hat{\phi}^{n-1}(z).$$
So by Lemma \ref{units-idempotents}, $y^nex^n$  is a unit in
$\hat{\phi}^n(e)\hat{R}_\tau$ for all $0\leq  n<l.$ 
As $\hat{R}_\tau=(e,\hat{\phi}(e),\dots, \hat{\phi}^{l-1}(e))$, we conclude
that the elements $y^nex^n, 0\leq n<l$ generate
$\hat{R}_\tau.$ Hence $\hat{H}=\hat{H}e\hat{H}.$

For any $r\in \hat{R}_{\tau},$ we have that
$erx^ne=rx^n\hat{\phi}^n(e)e=0$ unless $l$ divides $n.$ Similarly, we
check that $e\hat{R}_\tau y^ne=0$ for all $n$ unless $l$ divides $n$, giving the
desired equality for $e\hat{H}e.$

Finally, we have for $0<n<l:$
\[
y^{l-n}e(ex^{l}e)=y^{l-n}x^{l}e=x^n\hat{\phi}^n(z\hat{\phi}(z)\cdots
\hat{\phi}^{l-n-1}(z))e=x^{n}\hat{\phi}^n(z)\cdots\hat{\phi}^{l-1}(z)e.
\]
Now recall that $\hat{\phi}^i(z)e$ is a unit in $e\hat{R}_\tau$.
So, $x^{n}e$ belongs to the $e\hat{H}e$-span of $\lbrace e,
ye,\dots, y^{l-1}e\rbrace$. 
Also, since $y^{lm+n}e=y^ne(ey^{lm}e)$, it follows that all $y^ne$ belong
to this span. Hence the elements 
$\lbrace e, ye,\dots, y^{l-1}e\rbrace$ generate $\hat{H}e$ over $e\hat{H}e.$ 

It remains to show that $\lbrace e, ye,\dots, y^{l-1}e\rbrace$ are linearly
independent over $e\hat{H}e.$ Indeed, since $\hat{H}e$ is a graded module
over $e\hat{H}e$ and $\lbrace e, ye,\dots, y^{l-1}e\rbrace$ are
homogeneous elements of degree $0,-1,\dots, -(l-1)$, and since
the degrees of homogeneous elements of 
$e\hat{H}e$ are multiples of $l$, it suffices to check that if $y^ied=0$ with $d\in e\hat{H}e$ homogeneous, then $d=0.$ This is immediate since it implies $0=x^iy^ied=\prod_{j=1}^{i}\hat{\phi}^{-j}(z)ed$, and $\prod_{j=1}^{i}\hat{\phi}^{-j}(z)e$ is a unit of $e\hat{R}_\tau$ by Lemma \ref{units-idempotents}, since $\hat{\phi}^{-j}(z)\equiv\hat{\phi}^{l-j}(z)\pmod{b}$. Hence we are done.
\end{proof}
Next, we study the structure of the ring $e\hat{H}e$. This is done in the next lemma, which is essentially a tautology.

\begin{lemma}\label{eHe}
Let $x'$ and $y'$ be the standard generators of
$H(e\hat{R}_\tau,\hat{\phi}^l,e\tau)$. The map which is the identity on
$e\hat{R}_\tau$ and sends
\[
x'\longmapsto ex^l,\qquad y'\longmapsto ey^l
\]
is an isomorphism of algebras $H(e\hat{R}_\tau,\hat{\phi}^l,e\tau)\cong e\hat{H}e$.
\end{lemma}

\begin{proof}
We must verify that $ex^l, ey^l$ generate $e\hat{H}e$ over
$e\hat{R}_\tau$ and satisfy the corresponding relations of a GWA. This is more
or less straightforward from the previous result.
\end{proof}

To summarize, Morita theory gives an equivalence between the categories
of $H(\hat{R}_{\tau},\hat{\phi},z)$-modules and $e\hat{H}e$-modules.
Using the above identification of $e\hat{H}e$ with
$H(e\hat{R}_\tau,\hat{\phi}^l,e\tau)$, this equivalence is given
by the functor
\[
F:H(\hat{R}_{\tau}, \hat{\phi}, z)\text{-mod}\to H(e\hat{R}_{\tau},\hat{\phi}^l,
e\tau)\text{-mod},\quad F(M)=eM.
\]
The inverse functor is given by 
\[
G:H(e\hat{R}_{\tau},\hat{\phi}^l, e\tau)\text{-mod}\to H(\hat{R}_{\tau},
\hat{\phi}, z)\text{-mod},\quad G(N)=\hat{H}e\otimes_{e\hat{H}e}N.
\]
Combining this with Lemma \ref{category-O}, we get a
functor (still denoted by $F$)
\[
F:\mathcal{O}(H(R,\phi,z))\to H(e\hat{R}_{\tau},\hat{\phi}^l,e\tau)\text{-mod}.
\]

It follows immediately that the image of $F$ is in fact in the category
$\mathcal{O}.$ By Lemma \ref{units-idempotents}, we have $ez=u(e\tau)$ for a central
unit $u\in(e\hat{R}_\tau)^*$. Hence, replacing $y'$ by $uy'$ (and keeping $x'$ and the coefficient ring),
we may identify $H(e\hat{R}_\tau,\hat{\phi}^l,e\tau)$ with
$H(e\hat{R}_\tau,\hat{\phi}^l,ez)$.

Next, we claim that $e\hat{R}_{\tau}\cong \hat{R}_z$, the completion of
$R$ with respect to $z.$ Indeed, we have a ring homomorphism $f:R\to
e\hat{R}_\tau$ given by multiplication by $e.$ Since $ez\in (e\tau)$,
$f(z)$  is pro-nilpotent, and we may complete this
homomorphism to define $\hat{f}:\hat{R}_{z}\to e\hat{R}_\tau$. It suffices to
show that at each level $f_n:R/(z^n)\to e\hat{R}_\tau/(e\tau^n)$ is an
isomorphism. Indeed, since $\tau^n=z^n\cdots \phi^{l-1}(z^n),$  by the
Chinese remainder theorem we have
\[
R/(\tau^n)\cong (R/(z^n)) \times \cdots \times (R/(\phi^{l-1}(z)^n)).
\]
Denote by
\[ e_n\in R/(\tau^n) \]
the idempotent corresponding to the projection on $R/(z^n).$ 
But recall that $e$ is the lift of $e'=e_1$. Thus, in $R/(\tau)$, we have
\[
e\bmod \tau=e_n \bmod \tau=e_1.
\]
Using the uniqueness of the lifting idempotent modulo a
nilpotent ideal, we get that $e_n$ is the image of $e$ in $R/(\tau^n).$
This immediately gives that $f_n$ is an isomorphism and we are done.

So far, using the isomorphism $\hat{f}$, we have a functor  
\[
F:\mathcal{O}(H(R, \phi, z))\to \mathcal{O}(H(\hat{R}_z, \hat{\phi}^l, z)).
\]
Next, we show that $F$ actually lands in $\mathcal{O}(H(R, \phi^l, z)).$


\begin{lemma}
We have an equivalence of categories
\[
\mathcal{O}(H(R, \phi^l, z))\simeq\mathcal{O}(H(\hat{R}_z, \hat{\phi}^l, z)).
\]
\end{lemma}

\begin{proof}
The proof of this lemma is essentially the same as the proof of Lemma \ref{category-O}. We want to show $z$ acts locally nilpotently on any module in $\mathcal{O}(H(R,\phi^l,z))$.
Write $\phi^l(z)=z+ba$ for some $a\in Z(R).$ Then for all $n\in\mathbb{N}$ we have (using the first equality in the proof of
Lemma \ref{category-O} replacing $\phi$ by $\phi^l$)
\[
y^nx^n=\prod_{i=1}^{n}(z+ba_i),\quad a_i\in Z(R).
\]
As in the proof of Lemma \ref{category-O}, this implies that $z^n=\sum_{i=n-t+1}^nb_iy^ix^i$ for some $b_i\in Z(R),$ where $t$ is such that $b^t=0$.
Now, for any $m\in M\in \mathcal{O}(H(R,\phi^l, z))$ we
have that $y^kx^km=0$ for $k\gg 0$. This easily
implies that $z^km=0$ for $k\gg 0.$  So, $M$ belongs to
$\mathcal{O}(H(\hat{R}_z, \hat{\phi}^l, z)).$  
\end{proof}

By the above lemma, we now have the functor
\[
F:\mathcal{O}(H(R, \phi, z))\to \mathcal{O}(H(R ,\phi^l, z)).
\]
Under the identifications established above, it remains to show that the
functor
$G:H(\hat{R}_z, \hat{\phi}^l, z)\text{-mod}\to
H(\hat{R}_{\tau}, \hat{\phi}, z)\text{-mod}$
takes objects of $\mathcal{O}(H(R,\phi^l,z))$ to objects of
$\mathcal{O}(H(R,\phi,z))$. Since $\lbrace e, ye,\dots,
y^{l-1}e\rbrace$ is a basis of $\hat{H}e$ over $e\hat{H}e$ we have
\[
G(M)=\hat{H}e\otimes_{e\hat{H}e}M=\bigoplus_{i=0}^{l-1}t^iM.
\]
Here $t^iM$ denotes a copy of $M$, with the identification
$t^im=y^ie\otimes m$ for $m\in M$.

The following is a description of the action of $y, r\in
\hat{R}_{\tau}, x^l$ on $G(M).$ 
First note that  $[x^l, y]\in bH(R, \phi,z).$
Indeed, we have $\phi^{l-1}(z)=\phi^{-1}(z)+bd$ for some $d$, so
\[
yx^l=x^{l-1}\phi^{l-1}(z)=x^{l-1}(\phi^{-1}(z)+bd)=x^ly+bd', d'\in H(R,
\phi, z).
\]

Now using $x', y'$ for the standard generators in
$H(\hat{R}_z, \hat{\phi}^l,z)$ to avoid confusion, we have
\[
\begin{aligned}
r\cdot t^iM &= t^i(e\hat{\phi}^{-i}(r)M), &
x^l\cdot t^iM &= t^i(x'M)\pmod b,\\
y\cdot t^iM &= t^{i+1}M, &
y\cdot t^{l-1}M &= u^{-1}y'M.
\end{aligned}
\]
Now it is straightforward to see that if $x'$ acts locally nilpotently on
$M$, then so does $x^l$. 
%
Hence, $G$ sends $\mathcal{O}(H(R, \phi^l, z))$ 
to $\mathcal{O}(H(R, \phi,z)).$

Let $R_\phi[y]$ denote the twisted polynomial algebra over
$R$ by $\phi$, i.e.\ the subalgebra of $H(R, \phi, z)$ generated by $R,
y$. Then from the above it follows easily that $G(M)$ as a
module over $R_{\phi}[y]$ admits the following simple description. We
have a ring homomorphism, an analogue of the Frobenius,
$\Fr:R_{\phi^l}[y']\to  R_{\phi}[y]$ given by $\Fr(r)=r, r\in R$ and
$\Fr(y')=y^l.$
We identify $R_{\phi^l}[y']$   with the subring of $H(R, \phi^l, z)$
generated by $R, y'.$ 

For $M\in\mathcal{O}(H(R,\phi^l,z))$, denote by $M^{\mathrm{tw}}$ the module
$M$ on which $x'$ and $y'$ act by $\hat{\phi}^{-l}(u)x'$ and $u^{-1}y'$,
respectively; this is again an object of $\mathcal{O}(H(R,\phi^l,z))$.
Then it follows that $G(M)\in\mathcal{O}(H(R,\phi,z))$, as a module over
$R_{\phi}[y]$, is isomorphic to $R_{\phi}[y]\otimes_{R_{\phi^l}[y']}M^{\mathrm{tw}}$,
where $R_{\phi}[y]$ is viewed as a right $R_{\phi^l}[y']$-module via $\Fr$.


Finally, we show that $eM=M^{z^{\infty}}$ for any $M\in\mathcal{O}(H(R,
\phi, z)).$ Recall that for any $n$, $1-e\equiv a_nz^n \pmod{\tau^n}$  for some
$a_n\in \hat{R}_{\tau}.$ Then, by Lemma~\ref{category-O}, for any $m\in M$ there exists $n$ such that $\tau^nm=0$. If $z^nm=0$ then $(1-e)m=0$ and $m\in eM.$ So,
$M^{z^{\infty}}\subset eM.$ On the other hand, given $m\in eM$, by Lemma~\ref{units-idempotents}, we have $ez=u(e\tau)$ for a unit $u\in(e\hat{R}_\tau)^*$.  Then
$z^nm=(ez)^nm=u^n\tau^nm=0$ for $n\gg 0$ and we are done. This completes
the proof of the theorem. \qed

\section{Applications}

In this section we apply our main result to important families of
generalized Weyl algebras, such as noncommutative deformations of type A
Kleinian singularities (also known as classical generalized Weyl
algebras), and quantized Weyl algebras.

Let us recall the definitions.

\begin{defin}
Let $\bold{k}$ be a commutative ring and $v\in \bold{k}[h]$ be a nonzero
polynomial. The corresponding algebra $A(v)$ (a classical
GWA, or noncommutative deformation of type A Kleinian singularity) is defined
as the GWA $H(\bold{k}[h], \phi, v)$ where
$\phi:\bold{k}[h]\to \bold{k}[h]$ is the automorphism given by
translation by 1, so $\phi(f(h))=f(h-1).$ Hence $yx=v$ and $xy=v(h+1)$. If
we take $v=h$, then we recover the Weyl algebra $A_1(\bold{k}).$ Denote
by $A(v, p)$ the generalized Weyl algebra as above with $\phi$ replaced
by $\phi^p$ (translation by $p$).
\end{defin}

\begin{defin}
Let $\bold{k}$ be a unital commutative ring, $q\in\bold{k}^*$ and
$c\in\bold{k}$. Then the corresponding quantized Weyl algebra
$A_{q,c}(\bold{k})$ is defined as
\[
\bold{k}\langle x, y\rangle/(xy-qyx-c).
\]
It follows immediately that $A_{q,c}(\bold{k})$ is a generalized Weyl
algebra over $\bold{k}[h]$ with the automorphism $\phi^{-1}(h)=qh+c$
and the central element $z=h$.
\end{defin}

Next, we recall the definition of the $q$-integers:
$[n]_q=1+q+\cdots+q^{n-1}$ (which equals $\frac{1-q^n}{1-q}$ when $1-q$ is invertible).

\begin{cor}
Let $\bold{k}$  be a subfield of a commutative ring $S.$ 
Let $u\in\bold{k}$ be a primitive $l$-th root of unity and $b\in S$ a
nilpotent element, $q=u+b.$ Then the category $\mathcal{O}(A_{q,1}(S))$
is equivalent to the category $\mathcal{O}(A_{q^l, [l]_q}(S)).$ 
\end{cor}

\begin{proof}
We have that $\phi^{-i}(h)=q^ih+[i]_q.$  It follows that $(\phi^{-i}(h),
h)=1$ for $0<i<l$ and $\phi^{-l}(h)\equiv h \pmod{b}.$ Now the result follows
directly from Theorem \ref{main}.
\end{proof}

\begin{cor}\label{classical-corollary}
Let $\bold{k}$ be a commutative ring. Let $p$ be a prime
integer such that its image in $\bold{k}$ is nilpotent.
Let $v=\prod_i (h-\lambda_i), \lambda_i\in\bold{k}$ be a polynomial, such
that for any distinct pair $\lambda_i, \lambda_j$ we have
$\lambda_i-\lambda_j-n\in\bold{k}^*$ for all $0<n<p$ (for
example $v=h^m$, $m\geq 1$).
Then $\mathcal{O}(A(v))$  is equivalent to $\mathcal{O}(A(v,p)).$
\end{cor}

\begin{proof}
It follows immediately from our assumptions that $v(h+i)$  and $v$ are
coprime in $\bold{k}[h]$  for all $0<i<p,$ and $v(h+p)\equiv v(h)
\pmod{p}$. Now applying Theorem \ref{main} for $b=p,$ we are done.
\end{proof}

Our next result is an application of Theorem \ref{main} to the
representation theory of GWAs.

\begin{cor}
Assume in addition to the assumptions of Theorem \ref{main} that
$R$ is commutative and $\bold{k}\subset R$ is an algebraically closed field, such
that $R$ is a finitely generated $\bold{k}$-algebra; and
moreover,
$$\phi^l\equiv\mathrm{Id} \pmod{b}.$$ Then any simple module in category
$\mathcal{O}$  has dimension $l$ over $\bold{k}$.
\end{cor}

\begin{proof}
Let $M$ be a simple module in the category $\mathcal{O}$.
Since $b$ is a nilpotent central element of $H,$ it follows
that $bM=0.$  So, $M$ is a simple $\bar{H}=H(\bar{R}, \bar{\phi},
\bar{z})$-module where $\bar{R}=R/bR$ and $\bar{z}=z\bmod b,
\bar{\phi}=\phi\bmod b.$ In particular, $\bar{\phi}^l=\mathrm{Id}.$   Applying
Theorem \ref{main} to $\bar{H}$ , we see that $M\cong N^{\oplus l}$ (as $\bold{k}$-vector
spaces) where $N$ is a simple $H(\bar{R},\mathrm{Id}, \bar{z})$-module.
Now, $H(\bar{R}, \mathrm{Id}, \bar{z})$ is a commutative  finitely generated
$\bold{k}$-algebra, hence $N$ is $1$-dimensional by the Hilbert's
Nullstellensatz, and we are done.
\end{proof}

For example, let $\bold{k}$ be an algebraically closed field of characteristic $p>0$,
and let $v\in\bold{k}[h]$ satisfy the hypotheses of
Corollary~\ref{classical-corollary}. Since $\phi^p=\mathrm{Id}$,
the preceding corollary applies with $b=0$ and $l=p$, so every simple module
in $\mathcal O(A(v))$ has dimension $p$ over $\bold{k}$.


\begin{thebibliography}{KT}

\bibitem [B]{B}
V.~Bavula, {\em Generalized Weyl algebras and their representations},
Algebra i Analiz 4 no. 1 (1992), 75--97; English translation, St.
Petersburg Math. Journal 4 no. 1 (1993), 71--92.

\bibitem [H]{H}

T.~Hodges, {\em Noncommutative deformations of type-A Kleinian singularities},  J. Algebra 161 (1993), no. 2, 271--290. 

\bibitem [KT]{KT}
A.~Khare, A.~Tikaradze, {\em On category $\mathcal{O}$ over triangular generalized
Weyl algebras}, J. Algebra 449 (2016), 687--729.       

\bibitem [LR]{LR}
V.~Lunts, A.~Rosenberg, {\em Kashiwara Theorem for Hyperbolic Algebras},
preprint MPIM.

\bibitem [OV]{OV}
A.~Ogus, V.~Vologodsky, {\em Nonabelian Hodge theory in characteristic
$p$}, Publ. Math. IHES 106 (2007), 1--138.
\bibitem [R]{R}
A.~L.~Rosenberg, {\em Noncommutative Algebraic Geometry and Representations
of Quantized Algebras}, Mathematics and Its Applications 330, Kluwer Academic
Publishers, Dordrecht, 1995.
\bibitem [S]{S}
A.~Shiho, {\em Notes on generalizations of local Ogus-Vologodsky
correspondence}, J. Math. Sci. Univ. Tokyo 22 (2015), no. 3, 793--875.

\end{thebibliography}
\end{document}